\documentclass[12pt]{amsart}
\usepackage{amsmath, amssymb}
\usepackage{epsfig}
\usepackage{graphicx}

\newtheorem{theorem}{Theorem}

\newtheorem{definition}{Definition}

\newtheorem{lemma}{Lemma}

\begin{document}

\title[Vertex maps on graphs - Perron-Frobenius Theory ] {Vertex maps on graphs - Perron-Frobenius Theory }

\author{Chris Bernhardt}

\address{Fairfield
University\\Fairfield\\CT 06824}



\email{cbernhardt@fairfield.edu}

\begin{abstract} The goal of this paper is to describe the connections between Perron-Frobenius theory and vertex maps on graphs. In particular, it is shown how Perron-Frobenius theory gives results about the sets of integers that can arise as periods of periodic orbits, about the concepts of transitivity and topological mixing, and about horseshoes and topological entropy.
\end{abstract}

\maketitle
   
\section{Introduction}

In one-dimensional combinatorial dynamics, the basic starting point is to study maps on the interval. The fundamental and most well-known result is Sharkovsky's Theorem --- a theorem that concerns the ordering of the periods of periodic orbits. In the proof of this theorem it becomes clear that it is important to consider not just the period of the periodic orbit, but the way that the points in the orbit are permuted. Suppose that $f$ has a periodic orbit with period $n$. Let $x_1, x_2, \dots,  x_n$ denote the orbit where $x_i < x_{i+1}$ for $1 \leq i < n$. Then the points of the orbit are said to have permutation $\theta$ if $\theta$ is the permutation of the integers $1, \dots, n$ with $f(x_i)=x_{\theta(i)}$ for $1 \leq i \leq n$. Maps of the interval give a  partial ordering of cyclic permutations. 

Given a permutation $\theta$ of the integers from $1$ to $n$, construct a piecewise linear map $L_{\theta} : [1,n] \to [1,n]$ by $L(i)=\theta(i)$ for $1 \leq i \leq n$ and $L$ is linear on $[i, i+1]$ for $1 \leq i < n$. Such a map is often called {a connect-the-dots} map. In a sense, that we will make clear later, the map $L_{\theta}$ is the ``simplest" map of an interval that has a periodic point with permutation $\theta$. The study of the partial ordering of permutations given by maps of the interval can often be reduced to studying connect-the-dots maps.

One way of generalizing these ideas is to re-interpret the above ideas. Instead of regarding $[1,n]$ as a closed interval, think of it as a graph. The edges in the graph are the closed subintervals $[i, i+1]$ for $1 \leq i < n$ and the vertices of the graph correspond to the integers $1, \dots, n$. In this interpretation we have a combinatorial graph with $n-1$ edges and $n$ vertices. With this interpretation, it now seems natural to generalize our study to maps of connected graphs in which the vertices are permuted. We define a {\em vertex map} on a graph $G$ to be a map from $G$ to itself that permutes the vertices.

One of the basic tools for analyzing the sets of periodic points is to consider powers of the associated Markov matrix. This matrix is a non-negative matrix with integer entries. Perron-Frobenius theory gives a complete description of powers of non-negative matrices. The goal of this paper is to briefly describe vertex maps and how the Markov matrix is obtained. Then to briefly describe Perron-Frobenius theory. Finally to convert Perron-Frobenius language into language that is typically used in dynamical systems theory. In particular, we describe what Perron-Frobenius theory has to say about the sets of integers that can arise as periods of periodic orbits, about the concepts of transitivity and topological mixing, and about horseshoes and topological entropy.

Vertex maps form a special class of maps of graphs, but with that proviso, it should be noted that the hypotheses we assume throughout are fairly weak. We assume that the vertices are permuted, but we do not assume they form one periodic orbit. We do not assume the the underlying map has a certain homotopy type. (More specialized results that take into account the homotopy type of the underlying map are given in \cite{B1,B2,B3}.)

\section{Vertex maps}

A {\em vertex map} is a map $f$ from a graph $G$ to itself that permutes the vertices. First we define the underlying graphs $G$ that we consider and how the map $f$ is linearized.

\subsection{Graphs}

An {\it edge} is a space homeomorphic to the closed interval $[0,1]$. The boundary points of the edges are {\it vertices}. An edge is not allowed to have a vertex as an interior point. The intersection of two distinct edges is empty, consists of one vertex or of two vertices. We assume that we have a finite number $e$ of edges and $v$ of vertices. The {\it graph} is the union of vertices and edges. We assume that graphs are connected. We allow the possibility that there is more than one edge between the same two vertices. We do not allow the possibility that an edge connects a vertex to itself. Each edge has two distinct vertices.

\subsection{The linearization of the map}

Each edge in the graph is homeomorphic to the unit interval. We use the homeomorphism to define the distance between points in an edge and to give each edge unit length. A path consisting of $m$ edges is defined to have length $m$ in the obvious way. Suppose that an edge $E_i$ is mapped by $f$ to a path with $m$ edges, then there is a natural induced map $f^*:[0,1] \to [0,m]$. We will say that $f$ is {\em linear} on $E_i$ if $f^*$ is linear. In this case we will also define the {\em modulus of the slope} of $f$ on $E_i$ to be $m$. We will denote this by $mods(f,E_i)=m$.

We now define the {\em linearization} of the map $f$, which we will denote by $L_f$.  For all vertices $V \in G$, we define $L_f(V)=f(V)$. If $E_i$ is an edge with endpoints $V_1$ and $V_2$, we define $L_f$ to map $E_i$ linearly onto the unique reduced path from $f(V_1)$ to $f(V_2)$ that is obtained from $f(E_i)$.

More formally, let $[0,1]=I$, we define $L_f:G  \to G$ to be the {\em linearization} of $f$ if for each edge $E_i$ there is homotopy $h_i:E_i \times I \to G $ which has the following properties : $h_i(x,0)=f(x)$ for all $x \in E_i$; $h_i(x,1)=L_f(x)$ for all $x \in E_i$; $h_i(V_1, t)=f(V_1)=L_f(V_1)$ for all $t \in I$; $h_i(V_2, t)=f(V_2)=L_f(V_2)$ for all $t \in I$; and such that $L_f$ is linear on $E_i$.
If $L_f$ is the linearization of a map $f$, we say $L_f$ is \emph{linearized}.

In the literature, the maps that we are calling linearized are sometimes referred to as {\em linear models} for tree maps or {\em connect-the-dots} maps for interval maps, see \cite{ ALM}.

\section{Markov Graphs}

Given a  linearized map $L_f$ that permutes the vertices, we construct an {\em Markov Graph}, $MG$,  in the following way. The vertices of  $MG$ correspond to the edges of $G$. For each edge $E_i$ in the graph $G$ there is a vertex $E_i^\prime$ in  $MG$. A directed edge will be drawn from one $MG$ vertex, $E_j^\prime$ to another, $E_i^\prime$, if the edge $E_j$ in $G$ has a closed subinterval that maps under $L_f$ entirely onto $E_i$. A directed edge will be drawn for each such closed subinterval.  Though we will not use the term, in the literature, if $E_j$ contains a closed subinterval with image equal to $E_i$ it is said that $E_j$ {\em f-covers} $E_i$. Below we sketch standard results for Markov Graphs and refer the reader to \cite{ALM} for formal proofs that are stated in terms of $f$-covers.

Given a Markov Graph, we can define a sequence $E_{i_0}^\prime E_{i_1}^\prime\cdots E_{i_d}^\prime$ to be a \emph{walk} of length $d$ in the $MG$, where  each vertex $E_{i_k}^\prime$ in the sequence will have an edge connecting it to $E_{i_{k+1}}^\prime$ in the $MG$, or equivalently, there is a closed subinterval  of $E_{i_k}$ that gets mapped exactly onto $E_{i_{k+1}}$. We call a walk \emph{closed} if its first and last vertices are equal.

Closed walks are useful because if there is a closed walk of length $d$ from vertex $E_k^\prime$ to itself in the $MG$, then $L_f$ has a periodic point of period $d$. 
This is because if $E_{i_0}^\prime E_{i_1}^\prime \cdots E_{i_{d-1}}^\prime E_{i_d}^\prime$ is a closed walk with $E_{i_0}^\prime=E_{i_d}^\prime$
then we know that there is a subinterval $J_{d-1}$ in $E_{i_{d-1}}$ such that $L_f(J_{d-1})=E_{i_{d}}=E_{i_0}$. We can then find a subinterval $J_{d-2}$ in $E_{i_{d-2}}$ such that $f(J_{d-2})=J_{d-1}$. We proceed inductively until we obtain a subinterval $J_{0}$ of $E_{i_0}$ with the property that $L_f^d(J_0)=E_{i_{d}}=E_{i_0}$. Since $J_0 \subseteq E_{i_0}$ and $f^d(J_0)=E_{i_0}$, it follows that $L_f^d$ must have a fixed point in $E_{i_0}$. 

Conversely, if $x$ is a periodic point of $L_f$ with period $d$ and $x$ is not a vertex, then for each $i$ there is a unique edge $E$ in $G$ such that $f^i(x)\in E$. This sequence of edges gives a closed walk in the Markov graph.

Notice that the map from $L_f^d: J_0 \to E_{i_0}$ is linear. We can assign a {\em slope} to $L_f^d$ on the interval $J_0$ by taking the product of the moduli of the slopes of $L_f$ on the edges that make the closed walk and then multiplying by $\pm 1$ depending on whether $L_f^d$ is orientation preserving or reversing on $J_0$. More formally, we define the slope to be $orient(L_f^d,J_0)\prod^d_{k=1}(mods(L_f, E_{i_k})$, where $orient(L_f^d,J_0)=1$ if $L_f^d$ is orientation preserving $J_0$ and $orient(L_f^d,J_0)=-1$ if $L_f^d$ is orientation reversing on $J_0$.

There is a natural extension to closed walks. Given closed walk  $E_{i_0}^\prime E_{i_1}^\prime \cdots E_{i_{d-1}}^\prime E_{i_d}^\prime $ with $E_{i_0}^\prime=E_{i_d}^\prime$, we define its {\em slope} to be the slope of $L_f^d$ on the interval $J$, where $J$ is an interval such that $L_f^k(J) \subseteq E_{i_k}$ for $0 \leq k \leq d$.

We state some easy consequences of the above definitions.

\begin{lemma}
Let $G$ be a graph, $L_f$ a linearized vertex map and $MG$ its  Markov Graph. Let $E_{i_0}^\prime E_{i_1}^\prime \cdots E_{i_{d-1}}^\prime E_{i_d}^\prime$ be a closed walk in  $MG$ with $E_{i_0}^\prime=E_{i_d}^\prime$. Let $J_0$ be a subinterval of $E_0$ with $L_f^k(J_0) \subseteq E_{i_k}$, for $1\leq k \leq d$ and $L_f^d(J_0)=E_{i_0}$. 

\begin{enumerate}

\item If the slope of $L_f^d$ on $J_0$ is not $1$, then $L_f^d: J_0 \to E_{i_0}$ has a unique fixed point.

\item If the slope of $L_f^d$ on $J_0$ is negative, then $L_f^d: J_0 \to E_{i_0}$ has a unique fixed point that is in the interior of $J_0$.

\item If the slope of $L_f^d$ on $J_0$ is $1$, then  $J_0=E_{i_0}$ and $L_f^d: J_0 \to E_{i_0}$ is the identity map.

\end{enumerate}

\end{lemma}
 
As noted above, closed walks in the Markov Graph give us information about the periodic points of the linearized map $L_f$. Suppose that $L_f$ is the linearization of the vertex map $f$, then  these closed walks also give us information about periodic points of $f$. This follows from the fact that if there is a subinterval of $E_{i_k}$ that gets mapped exactly onto $E_{i_{k+1}}$ by $L_f$, then there must be a subinterval of $E_{i_k}$ that gets mapped exactly onto $E_{i_{k+1}}$ by $f$. With these observations we can re-write the previous lemma with the focus on $f$.

\begin{lemma}
Let $G$ be a graph, $L_f$ a linearized vertex map and $MG$ its Markov Graph. Let $E_{i_0}^\prime E_{i_1}^\prime \cdots E_{i_{d-1}}^\prime E_{i_d}^\prime$ be a closed walk in the $MG$ with $E_{i_0}^\prime=E_{i_d}^\prime$. Let $J_0$ be a subinterval of $E_0$ with $L_f^k(J_0) \subseteq E_{i_k}$, for $1\leq k \leq d$ and $L_f^d(J_0)=E_{i_0}$.

\begin{enumerate}
\item There exists   a subinterval $I_0$ of $E_0$ with $f^k(I_0) \subseteq E_{i_k}$, for $1\leq k \leq d$ and $f^d(I_0)=E_{i_0}$. 
\item There exists at least one point in $I_0$ that is fixed by $f^d$.
\item If the slope of $L_f^d$ on $J_0$ is negative, then $f^d: I_0 \to E_{i_0}$ has a fixed point  in the interior of $I_0$.

\end{enumerate}
\end{lemma}

\section{ The Markov matrix associated to $L_f$}

In this section we define a matrix that encodes much of  the information given by $L_f$ or, equivalently by the $MG$. This matrix will all be of size $e \times e$ where $e$ is the number of edges in the graph $G$ or, equivalently, the number of vertices in the $MG$. 

Given an $MG$ with $e$ vertices we will define the {\em Markov matrix}, $M(L_f)=M$, to be the matrix that has entry $M_{i,j}$ equal to the number of directed edges that go from $E_j^\prime$ to $E_i^\prime$ in the Markov graph.

The Markov matrix is a standard tool in analyzing combinatorial maps, both \cite{ALM} and \cite{BC} are excellent references. (These two books are also good references for the dynamical system terminology introduced in the last two sections of this paper.)

The basic result concerning powers of these matrices is stated in the lemma below. It is first stated in terms of the Markov graph and then the equivalent statement in terms of the underlying graph $G$ is given.

 \begin{lemma}
 Given a graph $G$ and a linearized map $L_f$. Denote its Markov matrix by $M$.
 
 \begin{enumerate}
 \item For any $1 \leq i,j \leq e$, $M^k_{i,j}$ equals the number of paths in $MG$ with length $k$ that begin at $E_j^\prime$ and end at $E_i^\prime$.

 \item There are $M^k_{i,j}$
  closed subintervals of edge $E_j$ in $G$ that have disjoint interiors and such that the image of each subinterval under $L_f^k$ is $E_i$.  \end{enumerate}
 \end{lemma}
 
 %Stuff about M and lengths of paths 

\section{Perron-Frobenius Theory}

Much of this section is standard directed graph theory. A good book that contains Perron-Frobenius Theory is  by Lind and Marcus \cite{LM}. We use their notation. 

\begin{definition}

A square real matrix is {\em non-negative} if each entry is non-negative.
\end{definition}

\begin{definition}

A non-negative matrix $M$ is {\em irreducible} if for each $i,j$ corresponding to an entry in the matrix, there exists a positive integer $k$ such that $(M^k)_{ij}>0$.
\end{definition}

\begin{definition}

A non-negative matrix $M$ is {\em primitive} if there exists a positive integer $N$ such that for any $k \geq N$ every entry of $M^k$ is positive.
\end{definition}

\subsection{Irreducible components}

Given a Markov graph we say that an vertex $E_i^\prime$ {\em communicates} with $E_j^\prime$ if there is a walk from $E_i^\prime$ to $E_j^\prime$ and a walk from $E_j^\prime$ to $E_i^\prime$. Communication is an equivalence relation that partitions the vertices of $MG$ into equivalence classes. Given two communicating equivalence classes $C_i$ and $C_j$ we say $C_i<C_j$ if there exists a walk from a vertex in $C_i$ to $C_j$. We can use this ordering to re-label the $k$ equivalence classes $C_1, C_2, \dots , C_k$, with the property that if $C_i<C_j$ then $i<j$.

We also relabel the vertices of $MG$ so that if $E_l^\prime \in C_i$ and $E_m^\prime \in C_j$, then $l < m$ if $i<j$. With this ordering, the Markov matrix, $M$ is given by
\[M=\left [ \begin{matrix} 
A_1 &0 &0 &\dots &0\\
* &A_2 &0 &\dots &0\\
* &* &A_3 &\dots & 0\\
\vdots &\vdots &\vdots &\ddots &\vdots\\
* &* &* &\dots &A_k
\end{matrix} \right ] \]

We let $MG_i$ denote the subgraph of $MG$ that consists of vertices in $C_i$ and has edges of $MG$ for which both the initial and terminal vertices belong to $C_i$.

The graphs $MG_i$ are called an {\em irreducible components of $MG$} and the matrices $A_i$ are the {\em irreducible components} of $M$. 

If we let $\chi_A(t)$ denote the characteristic polynomial of a matrix $A$ then 

\[\chi_M(t)=\chi_{A_1}(t)\chi_{A_2}(t) \dots \chi_{A_k}(t).\]

Consequently, at least one of the irreducible components must have the same spectral radius as $M$.
We state this as a theorem.

\begin{theorem}
Given an $MG$ we can partition it into communicating classes. Let $M$ denote the Markov matrix associated to $MG$ and let $\lambda_M$ denote its spectral radius. Then there exists a class $C_I$ such that its associated matrix $M_I$ is irreducible and has spectral radius  $\lambda_M$.

\end{theorem}
\subsection{Structure of irreducible components -- Frobenius form}

We will now study how the irreducible components are structured. We will let $C_I$ denote an irreducible component and $M_I$ its corresponding Markov matrix. 

Given an irreducible component of a directed graph, the greatest common divisor of the lengths of all the closed walks is called the {\em period} of the irreducible component. We use the period to decompose $M_I$ into primitive submatrices. If the period is $1$, then $M_I$ is primitive.

Given an irreducible component of a directed graph with period $p$, we pick a vertex $E^\prime$ and group the vertices into $p$ sets by \[P_i=\{E_j^\prime:\text{there is a walk from $E^\prime$ to $E_j^\prime$ of length } n \text{ with } n \equiv i\mod p\}.\]

This gives a partition of the vertices in $C_I$. We will call these classes {\em distance mod p} classes.

We can relabel the vertices so that if $E_l^\prime \in P_i$ and $E_m^\prime \in P_j$, then $l < m$ if $i<j$ and such that the vertices in $P_i$ are mapped into $P_{i+1\mod p}$ for $1 \leq i \leq p$. With this relabeling the Markov matrix has the form:

\[M_I=\left [ \begin{matrix} 
0 &A_1 &0 &0 &\dots &0\\
0 &0&A_2 &0 &\dots &0\\
 0&0&0&A_3 &\dots & 0\\
\vdots &\vdots &\vdots &\vdots &\ddots &\vdots\\
0&0&0&0&\dots &A_{p-1}\\
A_p&0&0&0&\dots &0
\end{matrix} \right ] \]

When the Markov matrix of the component is raised to the period of the component we obtain a block diagonal matrix:

\[M_I^p=\left [ \begin{matrix} 
D_1 &0 &0  &\dots &0\\
0 &D_2&0  &\dots &0\\
 0&0&D_3&\dots & 0\\
\vdots &\vdots &\vdots &\ddots &\vdots\\
0&0&0&\dots &D_p\\

\end{matrix} \right ] \]

It need not be the case that each of the $P_i$ contains the same number of vertices. This means that in general the matrices $A_i$ are rectangular and that the square matrices $D_i$ need not have the same size. However, the matrices $D_i$ are all primitive. Given $D_i$ and $D_j$ they may have different sizes and so different number of eigenvalues. However they only differ on the number of zero eigenvalues --- their sets of non-zero eigenvalues are equal and so they have the same spectral radius.

We summarize these results as:

\begin{theorem} Let $C_I$ be a communicating component of a Markov graph with associated matrix $M_I$. Let $p$ denote its period. We can partition the vertices into distance mod $p$ classes. For any  such class $P_i$ with associated matrix, let $D_i$ denote the submatrix of $M_I^p$ that corresponds to $P_i$. Then for each $i$, $D_i$ is primitive.

Let $\lambda_{D_i}$ denote the spectral radius of $D_i$ and $\lambda_{M_I}$ denote the spectral radius of $M_i$. Then for each $i$, $\lambda_{D_i}=\lambda_{M_I}^p$.

\end{theorem}

\subsection{Structure of primitive components -- Perron form}

If a matrix $A$ is primitive, then it has a positive real eigenvalue $\lambda_{A}$ with the property that if $\lambda$ is any other eigenvalue then $|\lambda| < \lambda_{A}$. Associated to this eigenvalue there is a non-negative right eigenvector $R$ and any other positive right eigenvector is a scalar multiple of $R$.  There is also a non-negative left eigenvector $L$ associated to $\lambda_{A}$ and any other positive left eigenvector is a scalar multiple of $L$. 

We will think of $L$ and $R$ as matrices. Choose $L$ and $R$ such that $LR=[1]$. With this choice, let $H=RL$. (Equivalently, $R$ and $L$ are chosen so that their dot product is $1$, and  $H$ is the tensor product $R \otimes L$.) Then \[\lim_{k \to \infty} \left( \frac{1}{\lambda_{A}} A \right)^k=H.\]

In what follows we will need to calculate $\log \lambda_A$. We re-state the above as follows:

\begin{theorem} \label{special lambda} Suppose that $A$ is a primitive matrix with spectral radius $\lambda_A$. Then for any $i$, $j$, \[\lim_{k \to \infty} \frac{1}{k}\log [(A^k)_{i,j}]=\log \lambda_A.\]

\end{theorem}

We can combine the previous three theorems to obtain the following, where $E_i^\prime$ is a vertex in the communicating class with largest spectral radius and $p$ is the associated period.

\begin{theorem} \label{general lambda}
Given a Markov graph  with Markov matrix $M$ and spectral radius $\lambda_M$, 

\begin{enumerate}

\item there exists positive integers $i$ and $p$ such that

\[\lim_{k \to \infty} \frac{1}{pk}\log [(M^{pk})_{i,i}]=\log \lambda_M;\]

\item for any $j$,

\[\limsup_{k \to \infty} \frac{1}{k}\log [(M^{k})_{j,j}] \leq \log \lambda_M.\]

\end{enumerate}
\end{theorem}

\section{Perron-Frobenius and periods of orbits of $L_f$}

Lemmas $1$ and $2$ relate closed walks in the $MG$ to periodic orbits of $L_f$. A closed walk in the $MG$ will belong to one irreducible component. So to analyze closed walks we can restrict attention to the irreducible components that correspond to the communicating classes. 

We will be interested in the minimum periods of the periodic points of $L_f$. These correspond to the lengths of closed walks that are not repetitive i.e. not the concatenation of a shorter walk. We will denote the set of lengths of non-repetitive closed walks of an $MG$ by $Len(MG)$.

In what follows we will often need to talk about sets of positive integers that contain {\em all but finitely many positive integers}. We will write $k$ is $abfmpi$ to mean that $k$ can be any positive integer, except possibly finitely many.

\begin{lemma}
Let $MG$ be an  Markov graph with Markov matrix $M$ that is primitive.
\begin{enumerate}
\item If the spectral radius of $M$ is $1$, then \[Len(MG)=\{1\}.\]
\item If the spectral radius of $M$ is greater than $1$, then \[Len(MG)=\{k | k \text{ is }abfmpi\}.\]

\end{enumerate}

\begin{proof}
Markov matrices have integer entries.  If the matrix $M$ has size $e \times e$ then the fact that $M$ is primitive means that there is some positive integer $k$ such that $M^k$ has no zero entries. This means that each entry of $M^k$ must be greater than or equal to $1$. Let $A$ be the $e \times e$ matrix with each entry equal to $1$. The spectral radius of $A$ is $e$. The spectral radius of $M^k$ must be at least that of $A$. So the spectral radius of $M$ is at least $e^{1/k}$. From this we can deduce that if the spectral radius of $M$ is $1$, then it must be the $1 \times 1$ matrix with $1$ as its only entry. Clearly in this case $Len(MG)=\{1\}$.

If the spectral radius of $M$ is greater than $1$, then there must be a $N$ such that $M^k_{1,1}\geq 2$ for all $k \geq N$. There must be at least two closed walks of length $k$ from $E_1^\prime$ to itself for each $k \geq N$. This implies that there must be two closed non-repetitive walks from and to $E_1^\prime$ with lengths that are relatively prime. Let their lengths be $n$ and $m$. A standard result from elementary number theory tells us that any integer greater than $mn$ can be written in the form $am+bn$ with both $a>0$ and $b>0$. We can construct a non-repetitive walk of length $am+bn$ by going around the non-repetitive walk of length $m$ a total of $a$ times and then going $b$ times around the non-repetitive walk of length $n$.

\end{proof}

\end{lemma}

\begin{lemma}

Let $C_I$ denote a communicating class with corresponding  Markov graph $MG_I$ and Markov matrix $M_I$. Let $p$ denote the period of $C_I$. 

\begin{enumerate}
\item If the spectral radius of of $M_I$ equals $1$, then \[Len(MG_I)=\{p\}.\]
\item If the spectral radius of of $M_I$ is greater than $1$, then \[Len(MG_I)=\{kp | k \text{ is }abfmpi\}.\]

\end{enumerate}

\end{lemma}

\begin{proof}
As noted in the previous section $M_I^k$ has zeros down the main diagonal when $k$ is not a multiple of $p$. Consequently there are no closed walks of lengths that are not multiples of $p$.

We also know \[M_I^p=\left [ \begin{matrix} 
D_1 &0 &0  &\dots &0\\
0 &D_2&0  &\dots &0\\
 0&0&D_3&\dots & 0\\
\vdots &\vdots &\vdots &\ddots &\vdots\\
0&0&0&\dots &D_p\\

\end{matrix} \right ] \]
 where the $D_i$ are the matrices corresponding to the distance classes. The $D_i$ are all primitive and have the same spectral radius.

If the spectral radius of of $M_I$ equals $1$ then $M_I^p$ must be the $p \times p$ identity matrix. Consequently $M_I$ is a permutation matrix. Up to starting and ending vertex there is only one non-repetitive closed walk  in $MG_I$ and it has length $p$.

If the spectral radius of $M_I$ is greater than $1$, then the spectral radii of all the $D_i$ are all greater than $1$. The $D_i$ are primitive so the previous lemma and the fact that the $D_i$ come from $M_I^p$  tell us that $Len(MG_I)=\{kp | k \text{ is all but finitely many positive integers}\}$.

\end{proof}

The periodic points of $L_f$ are described by the closed walks in the $MG$. The only exception are the vertices of $G$. These may be periodic and not correspond to a closed walk. Combining all the above information gives:

\begin{theorem} Let $G$ be a graph with linearized vertex map $L_f:G \to G$. Let $Per(L_f)$ denote the set of minimal periods of periodic points of $L_f$. Then there exist positive integers $p_1, p_2 \dots, p_r$ and $q_1, q_2, \dots, q_s$ such that \[Per(L_f)=\left(\bigcup_{i=1}^s\{kq_i | k \text{ is }abfmpi\}\right) \cup \{p_1, p_2 \dots p_r\}  \]

If the spectral radius of its Markov matrix is $1$, then  \[Per(L_f)= \{p_1, p_2 \dots p_r\} . \]
\end{theorem}

\section{Dynamics and Perron-Frobenius Theory}

There are some concepts in dynamical systems that are closely related to the ideas of a matrix being irreducible or primitive. In this section we compare the various concepts.
\begin{definition}

Let $X$ be a compact space, and let $f:X \to X$ be continuous. The map $f$ is {\em topologically transitive} if for every pair of non-empty open sets $U$ and $V$ in $X$, there is a positive integer $k$ such that $f^k(U) \cap V$ is non-empty.

\end{definition}

\begin{definition}

Let $X$ be a compact space, and let $f:X \to X$ be continuous. The map $f$ is {\em topologically mixing} if for every pair of non-empty open sets $U$ and $V$ in $X$, there is a positive integer $N$ such that $f^k(U) \cap V$ is non-empty for every $k \geq N$.

\end{definition}

%topologically mixing implies transitive

The following lemma follows easily from the above definitions.

\begin{lemma} Given a graph $G$ and linearized map $L_f$ then:

\begin{enumerate}
\item $M(L_f)$ is non-negative.
\item If $L_f$ is transitive, then $M(L_f)$ is irreducible.
\item If $L_f$ is topologically mixing, then $M(L_f)$ is primitive.
\end{enumerate}
\end{lemma}

The converses of the second two statements are not true in general, but they are if we make some minor modifications to the hypotheses.  If $M(L_f)$ is irreducible and is not a permutation matrix, then $L_f$ is transitive. If $M(L_f)$ is primitive and $G$ has more than one edge, then $L_f$ is topologically mixing. We prove these two statements as Theorems \ref{top trans} and {\ref{top mix}. First we prove a lemma concerning walks with slope of either plus or minus one.

\begin{lemma}
Let $G$ be a graph with linearized map $L_f$. If $M(L_f)$ is irreducible and there exists a closed walk in the  Markov graph with slope $\pm 1$, then $M(L_f)$ is a permutation matrix.
 
\end{lemma}

\begin{proof}
Since there is a closed walk with slope $\pm 1$ there is a set of edges that gets permuted by $L_f$. Since $M(L_f)$ is irreducible, this set must include every edge.
\end{proof}

\begin{theorem} \label{top trans}
Given a graph $G$ and linearized map $L_f$. If $M(L_f)$ is irreducible and not a permutation matrix, then $L_f$ is topologically transitive.

\end{theorem}

\begin{proof}

Let $I$ denote a subinterval of an edge. The subinterval could be open, half-open or closed. Let $a$ and $b$ denote the endpoints of the subinterval. The orientation of the edge gives an orientation of the subinterval. We assume without loss of generality that $a<b$. For each $k \geq 0$, the subinterval gives a path on the graph from $L_f^k(a)$ to $L_f^k(b)$. We denote this path as $L_f^k(I)$. 

First consider the case where for every integer $k \geq 0$, the path $L_f^k(I)$ does not contain a vertex in its interior. It must be the case that for each integer $k \geq 0$, the image $L_f^k(I)$ is contained in exactly one edge. It must also be the case that this sequence of edges is eventually periodic. This corresponds to a closed walk in the  Markov graph with slope of either $1$ or $-1$. The previous lemmas rules out both cases. So for any subinterval $I$ of an edge there must be an $M$ such that the path $L_f^M(I)$ contains a vertex in its interior.

Given any open set $U$ contained in $G$, we will show that for any edge $E$ there exists an $N$ with  $E \subseteq L^N_f(U)$. This is enough to show that $L_f$ is transitive as any open $V$ must intersect an edge.

There are two cases to consider. First consider the case when $U$ contains a vertex. Denote it by $V_1$. In this case we can find a half-closed subinterval $I$ such that $I \subseteq U$, $I$ has $V_1$ the closed endpoint and such that $I$ lies entirely within one edge. By the argument above we know that the  there must be an $M$ such that the path $L_f^M(I)$ contains a vertex in its interior. The path $L_f^M(I)$ also contains the vertex $L_f^M(V_1)$ as an endpoint. Since $L_f$ is linear on edges, it must be the case that $L_f^M(I)$ contains an edge of $G$. Since $M(L_f)$ is irreducible, this image of this edge under some iteration of $L_f$ must contain $E$.

If $U$ does not contain a vertex, then there is some open subinterval $J \subseteq U$ such that $J$ is contained within one edge. By the argument above we know that there must be an $M$ such that the path $L_f^M(J)$ contains a vertex in its interior. Again, we will denote it $V_1$. Let $I$ denote a half-closed subinterval such that $I \subseteq L_f^M(J)$, $I$ has $V_1$ as the closed endpoint and such that $I$ lies entirely within one edge. Repeating the argument in the previous paragraph completes the proof.

\end{proof}

\begin{theorem}\label{top mix}

Given a graph $G$ with more than one edge and linearized map $L_f$, if $M(L_f)$ is primitive then $L_f$ is topologically mixing.

\end{theorem}

\begin{proof}

Given any open $U$ in $G$, the previous proof shows that for any edge $E$ there exists an $N$ with  $E \subseteq L^N_f(U)$. Since $M(L_f)$ is primitive there exits an $M$ such that $L_f^M(E)=G$.

\end{proof}

\section{Topological entropy and horseshoes}

Let $f:G \to G$ be a vertex map and $L_f$ its linearization. In the previous section we considered the cases when $L_f$ was topologically transitive and topologically mixing. However, these facts are not related to whether or not $f$ is topologically transitive or mixing. In this section we use information from the Markov matrix to obtain information about horseshoes and topological entropy. This will give information, not only about $L_f$, but about $f$.

\begin{definition}

Let $f: G \to G$ be a vertex map. Let $k$ and $n$ be positive integers. We say that an edge $E$ in $G$ has a $n$-horseshoe under $f^k$ if there exist non-empty, closed subintervals with pairwise disjoint interiors $J_1, J_2, \dots, J_n$ such that \[(J_1 \cup \dots \cup J_n) \subseteq (f^k(J_1)\cap \dots \cap f^k(J_n))\]

\end{definition}

Horseshoes of $L_f$ are closely related to the diagonal entries of powers of $M(L_f)$. The observation before Lemma $2$ shows that if an edge $E$ has an $n$-horseshoe under $L_f^k$, then $E$ will have  $n$-horseshoe under $f^k$. It should be noted that the converse is not true. It is possible for $f$ to have horseshoes that are eliminated by the linearization.

The following lemma is an immediate consequence of Lemmas $2$ and $3$.

\begin{lemma} \label{horseshoes}
Let $f:G \to G$ be a vertex map. Let $M$ be the Markov matrix of its linearization. For any $k$, if $(M^k)_{i,i}$ is positive, then the edge $E_i$ has a $(M^k)_{i,i}$-horseshoe under $f^k$.

\end{lemma}

The topological entropy of the map is a measure of the exponential growth rate of the number of periodic orbits. For a Markov graph we can measure the exponential growth rate of the number of walks of length $n$ or the number of closed walks of length $n$. In both cases, it follows from Theorem \ref{general lambda} that this number exists and is $\log \lambda$, where $\lambda$ is the spectral radius of associated matrix.

\begin{definition}
Given a linearized map $L_f$ its {topological entropy} , denoted $h(L_f)$, is $\log \lambda$, where $\lambda$ is the spectral radius of $M(L_f)$.

\end{definition}

We note that the topological entropy of any vertex map $f$ can be defined. The important result is that $h(f) \geq h(L_f)$. It is in this sense that $L_f$ is the simplest vertex map of all the vertex maps that both permute the vertices in the same way as $f$ and are homotopic to $f$. (A basic reference for standard results on topological entropy is \cite{ALM}.)

Theorem 5 tells us that the entropy of $L_f$ is zero if and only if its set of periods is finite.

The following results relate the entropy of $L_f$ to horseshoes of $L_f$. 
\begin{theorem}
Let $f:G \to G$ be a vertex map.  If there is an edge $E$ in $G$ that has an $n$-horseshoe under $L_f^k$, then $h(L_f)\geq \frac{1}{k}\log(n)$.

\end{theorem}

\begin{proof}
Let $E_i$ be an edge that has an $n$-horseshoe under $L_f^k$. Let $M$ denote $M(L_f)$. Then $(M^k)_{i,i}=n$. For all positive integers $l$ we must have $(M^{lk})_{i,i} \geq n^l$. So \[\limsup_{l \to \infty} \frac{1}{lk}\log (M^{lk})_{i,i} \geq \frac{1}{k} \log n.\]

The entropy of $L_f$ is the logarithm of the spectral value of $M$. Theorem \ref{general lambda} completes the proof.

\end{proof}

\begin{theorem}
Let $G$ be a graph, with more than one edge, with linearized map $L_f$ such that is $M(L_f)$ is primitive. Let $h$ denote the entropy of $L_f$. For any $0<s<h$ and any edge $E$ we can find positive integers $k$ and $n$ such that  $1/k \log(n)>s$ and $f^k$ has an $n$-horseshoe on $E$.

\end{theorem}

\begin{proof}
Let $\lambda$ denote the spectral radius of $M(L_f)=M$. Theorem \ref{special lambda} tells us that for any $i$, \[\lim_{k \to \infty} \frac{1}{k}\log [(M^k)_{i,i}]=\log \lambda =h.\] 
So given any $s<h$ we can find a value of $k$ such that $\frac{1}{k}\log [(M^k)_{i,i}]>s$. Lemma \ref{horseshoes} completes the proof.

\end{proof}

The same proof using Theorem \ref{general lambda} gives the following.

\begin{theorem}
Let $G$ be a graph, with linearized map $L_f$ with positive entropy $h$. For any $0<s<h$ there exists an edge $E$ and positive integers $k$ and $n$ such that  $1/k \log(n)>s$ and $f^k$ has an $n$-horseshoe on $E$.

\end{theorem}

 \end{document}